\documentclass[final,3p,times]{elsarticle}
\usepackage{latexsym,epsfig,amssymb,amsmath,amsthm,color,url}

\newtheorem{thm}{Theorem}[section]
\newtheorem{lem}[thm]{Lemma}

\renewcommand{\theenumi}{\roman{enumi}}

\theoremstyle{definition}
\newtheorem{defn}[thm]{Definition}

\theoremstyle{remark}

\journal{Statistics And Probability Letters}

\begin{document}

\begin{frontmatter}

\title{Ruin probabilities under Sarmanov dependence structure}

\author{Krishanu Maulik}
\address{Krishanu Maulik, \ Theoretical Statistics and Mathematics Unit, \ Indian Statistical Institute, \ 203, Barrackpore Trunk Road, Kolkata 700108, India.}
\ead{krishanu@isical.ac.in.}
\author{Moumanti Podder}
\address{Moumanti Podder, \ Courant Institute of Mathematical Sciences, \ New York University, \ 251 Mercer Street, New York, NY 10012, United States.}
\ead{mp3460@nyu.edu.}

\begin{abstract}
Our work aims to study the tail behaviour of weighted sums of the form $\sum_{i=1}^{\infty} X_{i} \prod_{j=1}^{i}Y_{j}$, where $(X_{i}, Y_{i})$ are independent and identically distributed, with common joint distribution bivariate Sarmanov. Such quantities naturally arise in financial risk models. Each $X_{i}$ has a regularly varying tail. With sufficient conditions similar to those used by \cite{denisov:zwart:2007} imposed on these two sequences, and with certain suitably summable bounds similar to those proposed by \cite{hazra:maulik:2012}, we explore the tail distribution of the random variable $\sup_{n \geq 1}\sum_{i=1}^{n} X_i \prod_{j=1}^{i}Y_{j}$. The sufficient conditions used will relax the moment conditions on the $\{Y_{i}\}$ sequence.
\end{abstract}

\begin{keyword}
Regular variation \sep product of random variables \sep ruin probabilities \sep Sarmanov distribution.

\end{keyword}

\end{frontmatter}

\section{Introduction} \label{sec:intro}
Regularly varying distributions find several applications in areas of actuarial and financial mathematics, in the analysis of random coefficient linear processes such as ARMA and FARIMA, and in stochastic difference equations. We refer to \cite{nyrhinen:2012} for the study of the insurance ruin problem. The development of the capital is described as the solution to a stochastic difference equation. The net losses over the years are independent and identically distributed with regularly varying tail. \cite{yang:wang:2013} consider a discrete-time risk model with dependent insurance and financial risks. If $X_{n}$ denotes the insurance risk and $Y_{n}$ the financial risk or the stochastic discount factor in time $n$, then
\begin{equation} \label{discount time n}
S_{n} = \sum_{i=1}^{n} X_{i} \prod_{j=1}^{i} Y_{j}
\end{equation}
represents the stochastic discount value of aggregate net losses up to time $n$. In \cite{yang:wang:2013}, the finite and infinite time ruin probabilities are analyzed.

A random variable $X$ with tail distribution $\overline{F}$ is said to be regularly varying with index $-\alpha$, with $\alpha > 0$, if $\overline{F}(xy) \sim y^{-\alpha} \overline{F}(x)$ as $x \rightarrow \infty$, for all $y > 0$. This is denoted by $X \in RV_{-\alpha}$. Let $\{X_{n}, n \geq 1\}$ be a sequence of independent and identically distributed random variables with regularly varying tails, and $\{\Theta_{n}, n \geq 1\}$ be another sequence of random variables, not necessarily independent of $\{X_{n}\}$. The almost sure convergence and tail behaviour of $\sup_{n \geq 1} \sum_{i=1}^{n} \Theta_{i} X_{i}$ has been studied in the literature. Here and later, for two positive functions $a(x)$ and $b(x)$, we write $a(x) \sim b(x)$ as $x \rightarrow \infty$ if $\lim_{x \rightarrow \infty} a(x)/b(x) = 1$.

The study of the almost sure finiteness of the infinite sum $S_{\infty} = \sum_{i=1}^{\infty} X_{i} \prod_{j=1}^{i} Y_{j}$ has been a topic of sustained interest in the literature. The general problem has been addressed in \cite{hult:samorodnitsky:2008} for the case when the sequences $\left\{X_{i}\right\}$ and $\left\{Y_{i}\right\}$ are independent and $\{X_{i}\}$ an i.i.d.\ regularly varying sequence. See also \cite{fougeres:mercadier:2012} and \cite{yang:hashorva:2013}.

We address our problem in two parts: first we analyze the behaviour of the product, and then the sum. The main result in this direction is given in \cite{breiman:1965}, which proves that if $X \in RV_{-\alpha}$ and $\Theta$ independent of $X$ satisfies $E[\Theta^{\alpha + \varepsilon}] < \infty$ for some $\varepsilon > 0$, then $\Theta X \in RV_{-\alpha}$ with $P[\Theta X > x] \sim E[\Theta^{\alpha}]P[X > x]$ as $x \rightarrow \infty$. This result was extended to finite and infinite sums in \cite{resnick:willekens:1991}. They showed that if $\{X_{i}\}$ and $\{\Theta_{i}\}$ are independent of each other, the $X_{i}$s are i.i.d $RV_{-\alpha}$, and the $\Theta_{i}$s satisfy some extra moment assumptions, then $P\left[\sum_{i=1}^{\infty} \Theta_{i} X_{i} > x \right] \sim P[X_{1} > x] \sum_{i=1}^{\infty} E[\Theta_{i}^{\alpha}] \quad \text{as } x \rightarrow \infty$.

\cite{denisov:zwart:2007} replaced the extra moment assumptions with other sufficient conditions so that $P[\Theta X > x] \sim E[\Theta^{\alpha}]P[X > x]$ as $x \rightarrow \infty$. This was again extended to the finite and infinite sum case by \cite{hazra:maulik:2012}. Motivated by the ruin model of \cite{nyrhinen:2012} above, we restrict ourselves to the setup where $\Theta_{i} = \prod_{j=1}^{i} Y_{j}$, for i.i.d. $Y_{j}$.

We consider the finite time ruin probability by time $n$, given by
\begin{equation} \label{finite time ruin}
\Psi(x, n) = P\left[\max_{1 \leq k \leq n} S_{k} > x \right],
\end{equation}
and the infinite time ruin probability by
\begin{equation} \label{infinite time ruin}
\Psi(x) = P\left[\sup_{n \geq 1}S_{n} > x \right].
\end{equation}

\subsection{Some useful classes of distributions}
While classically, the insurance risk $\{X_{n}\}$ and discount factor $\{Y_{n}\}$ are assumed to be independent, \cite{yang:wang:2013} assumed that each $(X_{i}, Y_{i})$ follows a bivariate Sarmanov distribution, which is defined as follows.
\begin{defn} \label{sarmanov definition}
The pair of random variables $(X,Y)$ is said to follow a bivariate Sarmanov distribution, if
$$P(X \in dx, Y \in dy) = (1+\theta \phi_{1}(x)\phi_{2}(y))F(dx)G(dy), \quad x \in \mathbb{R}, y \geq 0,$$
where the kernels $\phi_{1}$ and $\phi_{2}$ are two real valued functions and the parameter $\theta$ is a real constant satisfying
$$E\{\phi_{1}(X)\} = E\{\phi_{2}(Y)\} = 0 \qquad \text{and} \qquad 1+\theta \phi_{1}(x)\phi_{2}(y) \geq 0, \quad x \in D_{X}, y \in D_{Y},$$ where $D_{X} \subset \mathbb{R}$ and $D_{Y} \subset \mathbb{R}^{+}$ are the supports of $X$ and $Y$, with marginals $F$ and $G$ respectively.
\end{defn}

This class of bivariate distributions is quite wide, covering a large number of well-known copulas such as the Farlie-Gumbel-Morgenstern (FGM) copula, which is recovered by taking $\phi_{1}(x) = 1 - 2F(x)$ and $\phi_{2}(y) = 1 - 2G(y)$. We refer the reader to  \cite{lee:1996} for further discussion. A bivariate Sarmanov distribution is called proper if $\theta \neq 0$ and none of $\phi_{1}$ and $\phi_{2}$ vanishes identically. To study the dependence structure of Sarmanov distribution, we need to define the class of dominatedly tail varying distributions.
\begin{defn}
A random variable $X$ with distribution function $F$ is called dominatedly-tail-varying, denoted by $X \in \mathcal{D}$, if for all $0 < y < 1$,
$\limsup_{x \rightarrow \infty}{\overline{F}(xy)}/{\overline{F}(x)} < \infty$.
\end{defn}

It is traditional to study the tail of the product of random variables in terms of the Breiman's condition, which we strive to weaken. For that we need to state  definitions of certain useful classes of distributions.
\begin{defn}
A random variable $X$ is said to be long tailed and denoted by $X \in \mathcal{L}$ if $P[X > x] \sim P[X > x + y]$ as $x \rightarrow \infty$, for any $y$.
\end{defn}

\begin{defn}
A non-negative function $f$ is in the class $\mathcal{S}_{d}$ and called a subexponential density if $$\lim_{x \rightarrow \infty} \int_{0}^{x} \frac{f(x - y)}{f(x)} f(y) dy = 2 \int_{0}^{\infty} f(u) du < \infty.$$ If $f \in \mathcal{S}_{d}$ is such that $f(x) = P[U > x]$ for some random variable $U$, we say that $U \in \mathcal{S}^{*}$.
\end{defn}

\begin{defn}
A non-negative random variable $T$ is in class $\mathcal{S}(\gamma), \gamma \geq 0$, if as $x \rightarrow \infty$, we have
$$\frac{P[T > x + y]}{P[T > x]} \rightarrow e^{-\gamma y} \quad \text{and} \quad \frac{P[T + T' > x]}{P[T > x]} \rightarrow 2 E[e^{\gamma T}] < \infty,$$ where $T'$ is an i.i.d. copy of $T$. For $\gamma = 0$, we get the class $\mathcal{S}$ of subexponential distributions.
\end{defn}

The crucial property used by \cite{yang:wang:2013} is that the bivariate Sarmanov dependence is not very strong. For this, they further assumed that the generic bivariate Sarmanov random vector $(X, Y)$ satisfies
\begin{equation} \label{limit of phi}
X \in RV_{-\alpha} \quad \text{and} \quad \lim_{x \rightarrow \infty} \phi_{1}(x) = d_{1}.
\end{equation}
These assumptions will also be made throughout this paper.
If $(X, Y)$ is bivariate Sarmanov, then asymptotically, the product $XY$ has the same tail distribution as the product $X Y^{*}_{\theta}$ where $X$ and $Y^{*}_{\theta}$ are independent and $Y^{*}_{\theta}$ is obtained through a change of measure. It has the distribution function $G_{\theta}$ with \begin{equation} \label{twisted version}
G_{\theta}(dy) = P[Y^{*}_{\theta} \in dy] = (1 + \theta d_{1} \phi_{2}(y))G(dy).
\end{equation}

This is formalized in Lemma~3.1 of \cite{yang:wang:2013}, but we need a less generalized version given in Theorem \ref{almost:independent}.
\begin{thm} \label{almost:independent}
Assume that $(X,Y)$ follows a bivariate Sarmanov distribution and \eqref{limit of phi} holds. Let $X^{*}$ and $Y^{*}$ be two independent random variables identically distributed as $X$ and $Y$ respectively, i.e. having marginals $F$ and $G$ respectively. Let $\overline{H^{*}}(x) = P[X^{*}Y^{*} > x]$. If now $H^{*} \in \mathcal{D}$ and $\overline{G}(x) = o(\overline{H^{*}}(x)),$ then
$P[X Y > x] \sim P[X^{*}Y^{*}_{\theta} > x],$
where $X^{*}, Y_{\theta}^{*}$ mutually independent and $Y_{\theta}^{*}$ has distribution $G_{\theta}$ as defined in \eqref{twisted version}.
\end{thm}

\cite{yang:wang:2013} considered one of the conditions proposed by \cite{denisov:zwart:2007} on $(X, Y)$, and showed that
\begin{equation} \label{product:yang:wang}
P[X Y > x] \sim (E[Y^{\alpha}] + \theta d_{1} E[\phi_{2}(Y) Y^{\alpha}]) \overline{F}(x),
\end{equation}
In Section \ref{sec:product}, we show that \eqref{product:yang:wang} still holds under the remaining three conditions assumed by \cite{denisov:zwart:2007}.

Under the same condition as used in establishing \eqref{product:yang:wang}, \cite{yang:wang:2013} showed that the finite time ruin probability
\begin{equation} \label{finitesum:yang:wang}
\Psi(x, n) \sim \frac{1 - E[Y^{\alpha}]^{n}}{1 - E[Y^{\alpha}]} (E[Y^{\alpha}] + \theta d_{1} E[\phi_{2}(Y) Y^{\alpha}]) \overline{F}(x),
\end{equation}
where they used the convention that $(1 - E[Y^{\alpha}]^{n})/(1 - E[Y^{\alpha}]) = n$ when $E[Y^{\alpha}] = 1$. In section \ref{sec:finite sum}, we again extend \eqref{finitesum:yang:wang} under the remaining three conditions of \cite{denisov:zwart:2007}.

\cite{yang:wang:2013} showed that the infinite time ruin probability, assuming extra moments of $Y_{j}$s as in \cite{hazra:maulik:2012}, satisfies
\begin{equation} \label{infinitesum:yang:wang}
\Psi(x) \sim \frac{E[Y^{\alpha}] + \theta d_{1} E[\phi_{2}(Y) Y^{\alpha}]}{1 - E[Y^{\alpha}]} \overline{F}(x).
\end{equation}
In Section \ref{sec:infinite sum}, we prove \eqref{infinitesum:yang:wang} assuming only the conditions in \cite{denisov:zwart:2007}, and some uniform integrability assumptions.

\section{Product results} \label{sec:product}
We start this section with collecting the main product results of \cite{denisov:zwart:2007}. We first recall the complete characterization of slowly varying functions from Lemma~2.1 of \cite{denisov:zwart:2007} in our Lemma \ref{slowly varying categories}. We then state, in Theorem \ref{denisov:zwart}, the four sufficient conditions given in Propositions 2.1 through 2.3 of \cite{denisov:zwart:2007}.
\begin{lem} \label{slowly varying categories}
Let $X$ be nonnegative with tail distribution $\overline{F} \in RV_{-\alpha}$. We write $\overline{F}(x) = x^{-\alpha}L(x)$, where $L$ is slowly varying. In this case, $L$ must be one of the following forms:\\
\begin{enumerate}[(i)]
\item \label{(i)} $L(x) = c(x);$\\
\item \label{(ii)}$ L(x) = c(x)/P(V>\log x);$\\
\item \label{(iii)} $L(x) = c(x)P(U > \log x);$\\
\item \label{(iv)} $L(x) = c(x)P(U > \log x)/P(V > \log x);$\\
\end{enumerate}
where $U$ and $V$ are long tailed random variables and $\lim_{x \rightarrow \infty}c(x) = c \in (0, \infty)$.
\end{lem}

\cite{denisov:zwart:2007} introduced the following four conditions, referred here as (DZ) conditions, enough to ensure Breiman-type results.
\begin{thm}  \label{denisov:zwart}
Let $X$ be nonnegative with tail distribution $\overline{F} \in RV_{-\alpha}$, and $Y$ be independent of $X$, satisfying $E(Y^{\alpha}) < \infty$ and $P\{Y>x\} = o(P\{X>x\})$ as $x \rightarrow \infty$. We write $\overline{F}(x) = x^{-\alpha}L(x)$, $L$ slowly varying. Consider the following conditions:
\renewcommand{\theenumi}{DZ\arabic{enumi}}
\begin{enumerate}
\item \label{DZ1} $\lim_{x \rightarrow \infty} \sup_{y \in [1,x]} L(y)/L(x) < \infty$;\\
\item \label{DZ2} $L$ is of type (\ref{(iii)}) or (\ref{(iv)}) and $L(e^{x}) \in \mathcal{S}_{d}$ ;\\
\item \label{DZ3} $L$ is of type (\ref{(iii)}) or (\ref{(iv)}), $U \in \mathcal{S}^{*}$ and $P(Y > x) = o(x^{-\alpha}P[U > \log x])$; \\
\item \label{DZ4} When $E[U] = \infty$ or equivalently $E[X^{\alpha}] = \infty$, we define $m(x) = \int_{0}^{x} v^{\alpha}F(dv) \rightarrow \infty$, and assume $P(Y>x) = o(P[X>x]/m(x))$ and $\limsup_{x \rightarrow \infty}\sup_{\sqrt{x} \leq y \leq x}L(y)/L(x) < \infty$.\\
\end{enumerate}
If any one of the above conditions holds, then $P(XY > x) \sim E(Y^{\alpha})P(X>x)$.
\end{thm}

We need one more property of bivariate Sarmanov distribution, which is from Proposition 1.1 of \cite{yang:wang:2013}.
\begin{lem} \label{kernel:bounded}
Assume that $(X,Y)$ follows a proper bivariate Sarmanov distribution. Then there exist two positive constants $b_{1}$ and $b_{2}$ such that $|\phi_{1}(x)| \leq b_{1}$ for all $x \in D_{X}$ and $|\phi_{2}(y)| \leq b_{2}$ for all $y \in D_{Y}$.
\end{lem}

In the rest of the section, $X$ and $Y$ jointly follow bivariate Sarmanov, \eqref{limit of phi} holds and we additionally have
\begin{equation} \label{generic condition}
E[Y^{\alpha}] < \infty \quad \text{and} \quad P[Y > x] = o(P[X > x]) \Rightarrow \overline{G}(x) = o(\overline{F}(x)).
\end{equation}
We further assume that any one of the last three (DZ) conditions (\ref{DZ2}), (\ref{DZ3}) and (\ref{DZ4}) holds, and investigate the behaviour of the product $XY$.

Let $X^{*}$ and $Y^{*}$ be two mutually independent random variables with distribution functions $F$ and $G$ respectively. Let $\overline{H^{*}}(x) = P[X^{*}Y^{*} > x]$. Let $Y_{\theta}^{*}$ be the twisted version of $Y$ as given by \eqref{twisted version}. Observe that by Lemma \ref{kernel:bounded}
\begin{equation} \label{bound on marginal of twisted}
\overline{G_{\theta}}(x) = \int_{x}^{\infty}(1 + \theta d_{1} \phi_{2}(y))dG(y) \leq (1 + |\theta d_{1}| b_{2})\overline{G}(x) = o(\overline{F}(x)),
\end{equation}
and
\begin{equation*} 
E[(Y^{*}_{\theta})^{\alpha}] = \int_{0}^{\infty}y^{\alpha}(1 + \theta d_{1} \phi_{2}(y))dG(y) \leq (1 + |\theta d_{1}| b_{2})E[({Y^{*}}^{\alpha})] < \infty.
\end{equation*}
In Lemma \ref{DZ combined lemma} we show how any (DZ) condition that holds for $(X, Y)$, also extends to $(X^{*}, Y_{\theta}^{*})$. As a result, using Theorem \ref{denisov:zwart} we are able to conclude that $$P[X Y > x] \sim [E(Y^{\alpha}) + \theta d_{1} E(\phi_{2}(Y)Y^{\alpha})] \overline{F}(x).$$

\begin{lem} \label{DZ combined lemma}
Let any one of the four (DZ) conditions hold for $(X, Y)$. Then it also holds for $(X^{*}, Y_{\theta}^{*})$.
\end{lem}
\begin{proof}
Because each of the four (DZ) conditions involves only the properties of the marginal distributions of $X$ and $Y$, hence if they hold for $(X, Y)$, they also hold for $(X^{*}, Y^{*})$. For this same reason, (\ref{DZ1}) and (\ref{DZ2}) extend to $(X^{*}, Y_{\theta}^{*})$, and because of \eqref{bound on marginal of twisted}, (\ref{DZ3}) also extends to $(X^{*}, Y_{\theta}^{*})$.

We now consider (\ref{DZ4}). Because $X^{*}$ has the same distribution $F$ as $X$, hence $E[U] = \infty$ so that $m(x) \rightarrow \infty$, and $$\limsup_{x \rightarrow \infty}\sup_{\sqrt{x} \leq y \leq x}L(y)/L(x) < \infty.$$ From \eqref{bound on marginal of twisted},
$$\lim_{x \rightarrow \infty}\frac{P[Y^{*}_{\theta} > x]}{P[X^{*} > x]}m(x) \leq \lim_{x \rightarrow \infty}\frac{(1 + |\theta d_{1}| b_{2})\overline{G}(x)}{P[X^{*} > x]}m(x) = (1 + |\theta d_{1}| b_{2})\lim_{x \rightarrow \infty}\frac{\overline{G}(x)}{\overline{F}(x)}m(x) = 0.$$ Thus all aspects of condition (\ref{DZ4}) are satisfied by $(X^{*}, Y^{*}_{\theta})$.
\end{proof}

Summarizing everything, we have the following theorem.
\begin{thm} \label{final product thm}
The pair of random variables $(X, Y)$ jointly follow bivariate Sarmanov as given in Definition \ref{sarmanov definition}. $X$ nonnegative and $X \in RV_{-\alpha}$. We also assume $E[Y^{\alpha}] < \infty$, $P[Y > x] = o(P[X > x])$ and $\lim_{x \rightarrow \infty} \phi_{1}(x) = d_{1}$ exists. If any one of the three conditions (\ref{DZ2}), (\ref{DZ3}), (\ref{DZ4}) holds, then
$P[XY > x] \sim [E(Y^{\alpha}) + \theta d_{1} E\{\phi_{2}(Y) Y^{\alpha}\}] \overline{F}(x)$.
\end{thm}

Note that the similar result under the condition \eqref{DZ1} has already been proved in \cite{yang:wang:2013}.

\section{Finite Sum} \label{sec:finite sum}
We consider a sequence $\{(X_{i}, Y_{i})\}$ of independent and identically distributed random vectors, with the generic random vector $(X, Y)$ following bivariate Sarmanov and satisfying \eqref{limit of phi}. The finite time ruin probability under (\ref{DZ1}) has already been studied by \cite{yang:wang:2013}. We now show that if \eqref{generic condition} and any one of the three sufficient conditions (\ref{DZ2}), (\ref{DZ3}) and (\ref{DZ4}) is satisfied, then the same conclusion as in \eqref{finitesum:yang:wang} will hold.

Recall that the finite time ruin probability $\Psi(x, n) = P\left[\max_{1 \leq k \leq n} S_{k} > x \right]$. The first step is to prove
\begin{lem}
Assume that $\{(X_{i}, Y_{i}): i \in \mathbb{N}\}$ is an i.i.d.\ sequence of random vectors with the generic random vector $(X, Y)$ following a bivariate Sarmanov distribution as in Definition \ref{sarmanov definition}. Each $X_{i}$ is regularly varying with index $-\alpha$, and \eqref{limit of phi} holds.  If  $\overline{H_{i}}(x) = P[X_{i}\prod_{j=1}^{i}Y_{j} > x]$ and any of the conditions (\ref{DZ2}), (\ref{DZ3}) and (\ref{DZ4}) holds, then
\begin{equation} \label{sup like sum}
\Psi(x, n) \sim \sum_{i=1}^{n}\overline{H_{i}}(x).
\end{equation}
\end{lem}
The proof of \eqref{sup like sum} is similar to that of Theorem~4.1 in \cite{yang:wang:2013}, and hence we omit it.

The crucial step is then to establish that
\begin{equation} \label{induction}
\overline{H_{i}}(x) \sim \{E(Y^{\alpha})\}^{i-1}\overline{H}(x) \sim \{E(Y^{\alpha})\}^{i-1} [E(Y^{\alpha}) + \theta d_{1}E(\phi_{2}(Y) Y^{\alpha})]\overline{F}(x) = \{E(Y^{\alpha})\}^{i-1} E[{Y^{*}_{\theta}}^{\alpha}] \overline{F}(x)
\end{equation}
where $\overline{H}(x) = \overline{H_{1}}(x) = P[X_{1}Y_{1} > x]$.

We prove \eqref{induction} using induction on $i$. It holds for $i=1$ using Theorem \ref{final product thm}. Assume that \eqref{induction} holds for some $i \geq 1$ which implies that $\overline{H_{i}} \in RV_{-\alpha}$ since $\overline{F} \in RV_{-\alpha}$. Hence we can write $\overline{H_{i}}(x) = x^{-\alpha}L_{i}(x)$ where $L_{i}$ is a positive slowly varying function. Clearly this means that, by our induction hypothesis,
$L_{i}(x) \sim \{E(Y^{\alpha})\}^{i-1} E[{Y^{*}_{\theta}}^{\alpha}]L(x)$,
where $\overline{F}(x) = x^{-\alpha} L(x)$. Hence it is immediate that $L_{i}$ will have the same form as $L$, that is, the appropriate one from (\ref{(i)}) through (\ref{(iv)}) of Lemma \ref{slowly varying categories} holds. Since (\ref{DZ2}) and (\ref{DZ3}) involve only the asymptotic tail properties of $L$, they carry over to $L_{i}$ as well. We separately check the similar extension of the result for (\ref{DZ4}).

\begin{lem} \label{induction for DZ4}
If $(X, Y)$, or equivalently, $(F, G)$, satisfies (\ref{DZ4}), and \eqref{induction} holds for some $i \geq 1$, then the joint distribution $(H_{i}, G)$ also satisfies (\ref{DZ4}).
\end{lem}

\begin{proof}
Since, by induction hypothesis, we have $L_{i}(x)/L(x) \rightarrow \{E(Y^{\alpha})\}^{i-1} E[{Y^{*}_{\theta}}^{\alpha}]$, and \eqref{DZ4} holds for $L$, we have
$\limsup_{x \rightarrow \infty}\sup_{\sqrt{x} \leq y \leq x} {L_{i}(y)}/{L_{i}(x)} < \infty$.

Let us define $m_{i}(x) = \int_{0}^{x} t^{\alpha} dH_{i}(t)$. Observe that
$$m_{i}(x) = \alpha \int_{0}^{x} s^{\alpha - 1} \overline{H_{i}}(s) ds - x^{\alpha} \overline{H_{i}}(x), \qquad \text{and} \qquad
m(x) = \alpha \int_{0}^{x} s^{\alpha - 1} \overline{F}(s) ds - x^{\alpha} \overline{F}(x).$$
To check (\ref{DZ4}) for $(H_{i}, G)$ it is enough to check that $m_{i}(x)/m(x)$ is bounded. Observe that
\begin{align}
\limsup_{x \rightarrow \infty} \frac{m_{i}(x)}{m(x)} =& \limsup_{x \rightarrow \infty} \frac{1 - \frac{x^{\alpha} \overline{H_{i}}(x)}{\alpha \int_{0}^{x} s^{\alpha - 1} \overline{H_{i}}(s) ds}}{1 - \frac{x^{\alpha} \overline{F}(x)}{\alpha \int_{0}^{x} s^{\alpha - 1} \overline{F}(s) ds}} \cdot \frac{\int_{0}^{x} s^{\alpha - 1} \overline{H_{i}}(s) ds}{\int_{0}^{x} s^{\alpha - 1} \overline{F}(s) ds} \nonumber\\
=& 1. \limsup_{x \rightarrow \infty} \frac{\int_{0}^{x} s^{\alpha - 1} \overline{H_{i}}(s) ds}{\int_{0}^{x} s^{\alpha - 1} \overline{F}(s) ds} \leq \sup_{x > 0} \frac{\overline{H}_{i}(x)}{\overline{F}(x)} < \infty, \nonumber
\end{align}
where the second equality follows from Karamata's theorem. Hence, by the (\ref{DZ4}) condition on $L$,
$$\lim_{x \rightarrow \infty}\frac{\overline{G}(x)}{\overline{H_{i}}(x)}m_{i}(x) =\lim_{x \rightarrow \infty}\frac{\overline{G}(x)}{\overline{F}(x)}\frac{\overline{F}(x)}{\overline{H_{i}}(x)}\frac{m_{i}(x)}{m(x)} m(x) \leq \lim_{x \rightarrow \infty}\frac{\overline{G}(x)}{\overline{F}(x)}m(x) \lim_{x \rightarrow \infty} \frac{\overline{F}(x)}{\overline{H_{i}}(x)} \limsup_{x \rightarrow \infty}\frac{m_{i}(x)}{m(x)} = 0.$$
\end{proof}

Lastly, $Y_{1}$ is independent of $X_{i+1}Y_{i+1}Y_{i}...Y_{2}$ with distribution $H_{i}$. The appropriate (DZ) condition for $(H_{i}, G)$ gives
$$\overline{H_{i+1}}(x) = P[(X_{i+1}Y_{i+1}Y_{i}...Y_{2})Y_{1} > x] \sim E(Y^{\alpha})\overline{H_{i}}(x) \sim \{E(Y^{\alpha})\}^{i}\overline{H}(x).$$ This shows that the result \eqref{induction} holds for $i+1$ as well, and the induction is completed.

Summarizing, we now have the following theorem.
\begin{thm} \label{finite sum final thm}
Let $\{(X_{i}, Y_{i})\}$ be a sequence of independent and identically distributed random vectors, with the generic random vector $(X, Y)$ following bivariate Sarmanov as in Definition \ref{sarmanov definition}, with $X \in RV_{-\alpha}$. Suppose $E[Y^{\alpha}] < \infty, P[Y > x] = o(P[X > x])$ and $\lim_{x \rightarrow \infty} \phi_{1}(x) = d_{1}$. Let $\Psi(x, n)$ be as defined in \eqref{finite time ruin}. If any one of the conditions (\ref{DZ2}), (\ref{DZ3}) and (\ref{DZ4}) holds, then we have
\begin{equation*}
\Psi(x, n) \sim \frac{(1 - E[Y^{\alpha}]^{n}) \{E[Y^{\alpha}] + \theta d_{1} E[\phi_{2}(Y) Y^{\alpha}]\}}{(1 - E[Y^{\alpha}])} \overline{F}(x),
\end{equation*}
with the convention that $(1 - E[Y^{\alpha}]^{n})/(1 - E[Y^{\alpha}]) = n$ when $E[Y^{\alpha}] = 1$.
\end{thm}

\section{Infinite sum} \label{sec:infinite sum}
In this section, we consider again a sequence $\{(X_{i}, Y_{i})\}$ of i.i.d. random vectors with the generic random vector $(X, Y)$ jointly bivariate Sarmanov, with both \eqref{limit of phi} and \eqref{generic condition} satisfied. Additionally, we assume that $E[Y^{\alpha}] < 1$. Now we show that, if any of the four (DZ) conditions is also satisfied, along with some uniform integrability condition, then the same conclusion as \eqref{infinitesum:yang:wang} holds, that is
\begin{equation*}
\lim_{x \rightarrow \infty} \frac{\Psi(x)}{\overline{F}(x)} = \frac{E[Y^{\alpha}] + \theta d_{1} E[\phi_{2}(Y) Y^{\alpha}]}{1 - E[Y^{\alpha}]} = \frac{E[{Y^{*}_{\theta}}^{\alpha}]}{1 - E[Y^{\alpha}]},
\end{equation*}
where $Y_{\theta}^{*}$ is the twisted version of $Y$ given in \eqref{twisted version}.
The lower bound for $\Psi(x)/\overline{F}(x)$ follows immediately from a common argument for all the four (DZ) conditions:

For any $m \in \mathbb{N}$, using Theorem \ref{finite sum final thm}, or Theorem~4.1 of \cite{yang:wang:2013}, we get
$$\frac{\Psi(x)}{\overline{F}(x)} \geq \frac{\Psi(x, m)}{\overline{F}(x)} \sim \frac{1-\{E(Y^{\alpha})\}^{m}}{1-E(Y^{\alpha})} \cdot E[{Y^{*}_{\theta}}^{\alpha}],$$
and the desired lower bound now follows by letting $m \rightarrow \infty$.

For the upper bound we proceed as follows. Let $\zeta_{i} = \prod_{j=1}^{i-1}Y_{j}$ and $Z_{i} = X_{i}Y_{i}$. Observe that $Z_{i}$ and $\zeta_{i}$ are mutually independent. Then for any natural number $m$, any constant $0 < \delta < 1$ and any $x \geq 0$,
\begin{equation} \label{split sum}
P \left[ \sup_{1 \leq n < \infty} \sum_{i=1}^{n}Z_{i}\zeta_{i} > x \right] \leq P \left[\max_{1 \leq k \leq m} \sum_{i=1}^{k}Z_{i}\zeta_{i} > (1-\delta)x \right] + P \left[\sum_{i=m+1}^{\infty}Z_{i}\zeta_{i} > \delta x \right].
\end{equation}
Using Theorem~4.1 of \cite{yang:wang:2013} for (\ref{DZ1}), and Theorem \ref{finite sum final thm} for (\ref{DZ2}), (\ref{DZ3}) or (\ref{DZ4}), we have
$$P \left[\max_{1 \leq k \leq m} \sum_{i=1}^{k}Z_{i}\zeta_{i} > (1-\delta)x \right]=\Psi((1-\delta)x,m) \sim \frac{1-\{E(Y^{\alpha})\}^{m}}{1-E(Y^{\alpha})} E[{Y^{*}_{\theta}}^{\alpha}] \overline{F}((1-\delta)x).$$

Since $\overline{F} \in RV_{-\alpha}$, we have $\limsup_{x \rightarrow \infty} {P[\max_{1 \leq k \leq m} \sum_{i=1}^{k}Z_{i}\zeta_{i} > (1-\delta)x]}/{\overline{F}(x)} \leq \frac{E[{Y^{*}_{\theta}}^{\alpha}]}{1-E(Y^{\alpha})} (1-\delta)^{-\alpha}$.

We obtain the desired upper bound by making the second term of \eqref{split sum} arbitrarily small for suitably large $m$ and for all sufficiently large values of $x$, and finally letting $\delta \rightarrow 0$.
\begin{equation} \label{second term}
P\left[\sum_{i=m+1}^{\infty}Z_{i}\zeta_{i} > x \right] \leq  \sum_{i=m+1}^{\infty}P \left[Z_{i}\zeta_{i} > x \right] + P \left[\sum_{i=m+1}^{\infty}Z_{i}\zeta_{i} \mathbf{1}_{[Z_{i}\zeta_{i} \leq x]} > x \right].
\end{equation}

We bound the second term of \eqref{second term}, separately for $\alpha <1$ and $\alpha \geq 1$, arguing as in the proof of Theorem~4.2 in \cite{yang:wang:2013}. For $\alpha < 1$, we use Markov's inequality and for $\alpha \geq 1$ we use Minkowski's inequality. In both cases, using Karamata's theorem, we get a constant $C$ such that
\[\frac{P[\sum_{i=m+1}^{\infty}Z_{i}^{+}\zeta_{i} \mathbf{1}_{[Z_{i}^{+}\zeta_{i} \leq x]} > x]}{\overline{F}(x)} \leq \left\{\begin{array}{ll}
C \sum_{i=m+1}^{\infty} \frac{P[Z_{i}\zeta_{i}>x]}{\overline{F}(x)}& \mbox{if $\alpha < 1$,}\\
\sum_{i=m+1}^{\infty} \frac{P[Z_{i}\zeta_{i}>x]}{\overline{F}(x)} \\ + C[\sum_{i=m+1}^{\infty} (\frac{P[Z_{i}\zeta_{i}>x]}{\overline{F}(x)})^{\frac{1}{\alpha+\varepsilon}}]^{\alpha+\varepsilon}&\mbox{if $\alpha \geq 1$.}
\end{array}\right.\]

Then the upper bound will be established by showing that
${P[Z_{i}\zeta_{i}>x]}/{\overline{F}(x)} \leq B_{i}$
uniformly for all large values of $x$, that is, there exists $x_0$ such that for all $x>x_0$, we have $P[Z_{i}\zeta_{i}>x] \leq B_i \overline F(x)$ for all $i$. Here $B_{i}$ is a finite positive constant such that
\begin{equation} \label{summability}
\sum_{i=1}^{\infty}B_{i} < \infty \quad \text{for } \alpha < 1 \quad \text{and} \quad
\sum_{i=1}^{\infty}B_{i}^{\frac{1}{\alpha+\varepsilon}} < \infty \quad \text{for } \alpha \geq 1.
\end{equation}

For this, it will be sufficient to produce an upper bound for $P[Z_{i}\zeta_{i}>x]/P[Z_{i} > x]$ which satisfies \eqref{summability}. We split the ratio as follows:
$$\frac{P[Z_{i}\zeta_{i} > x]}{P[Z_{i} > x]} = \int_{(0, 1]} + \int_{(1, \infty)} \frac{P[Z_{i} >x/v]}{P[Z_{i} > x]}G_{i}(dv),$$ where $G_{i}$ is the distribution function of $\zeta_{i}$. As $x \rightarrow \infty$, the integrand converges to $v^{\alpha}$ uniformly in $v$ over the first interval and hence, for all large enough $x$,
$\int_{(0,1]}\frac{P[Z_{i} >x/v]}{P[Z_{i} > x]}G_{i}(dv) \leq 2E(\zeta_{i}^{\alpha})$.
The bound for the other integral is provided separately for the four (DZ) conditions. Recall that the (DZ) conditions are given in terms of the slowly varying function $L(x)=x^\alpha \overline F(x)$.

\begin{lem} \label{DZ1:infinite sum}
Let $\{(X_{n},Y_{n})\}$ be i.i.d.\ random vectors, with the generic random vector $(X, Y)$ jointly distributed as bivariate Sarmanov, and satisfying \eqref{limit of phi} and \eqref{generic condition}. Also, the (\ref{DZ1}) condition holds and $E[Y^{\alpha}] < 1$. Then
\begin{equation} \label{DZ1:second}
\int_{(1, \infty)} \frac{P[Z_{i} > x/v]}{P[Z_{i} > x]}G_{i}(dv)  \leq C' E(\zeta_{i}^{\alpha}),
\end{equation}
for some constant $C'$ independent of $i$, and for all sufficiently large $x$ uniformly in $i$.
\end{lem}

\begin{proof}
We have $\overline{H}(x) = P(Z_{i} > x) = P(X_{i}Y_{i} > x) = x^{-\alpha}L_{1}(x)$ where $L_{1}$ is slowly varying. Then
$\lim_{x \rightarrow \infty}L_{1}(x)/L(x) = E[{Y^{*}_{\theta}}^{\alpha}] \in (0, \infty)$. Thus $\limsup_{x \rightarrow \infty} \sup_{1 \leq y \leq x} L_{1}(y)/L_{1}(x)$ is finite.
We split the integral in \eqref{DZ1:second} over two intervals, $(1, x]$ and $(x, \infty)$. For the integral over $(x, \infty)$, for all $x$ large enough, uniformly in $i$, we have:
$$\int_{(x,\infty)}\frac{P[Z_{i} >x/v]}{P[Z_{i} > x]}G_{i}(dv)  \leq  \frac{P[\zeta_{i} > x]}{P[Z_{i} > x]}
\leq  \frac{E(\zeta_{i}^{\alpha})}{L_{1}(x)},$$ which is bounded above by a constant multiple of $E[\zeta_{i}^{\alpha}]$, the constant being independent of $i$.
For the integral over the range $(1, x]$, for all sufficiently large $x$ uniformly in $i$, we have
$$\int_{(1,x]} \frac{P[Z_{i} > x/v]}{P[Z_{i} > x]}G_{i}(dv) \leq \sup_{1 \leq y \leq x}\frac{L_{1}(y)}{L_{1}(x)}\int_{[1,x)}v^{\alpha}G_{i}(dv),$$
which is once again bounded above by a constant multiple of $E[\zeta_{i}^{\alpha}]$, the constant free of $i$.
\end{proof}

\begin{lem} \label{DZ2:infinite sum}
Assume that $\{(X_{i}, Y_{i}), i \geq 1\}$ are i.i.d.\ random vectors with the generic random vector $(X, Y)$ following a bivariate Sarmanov distribution, satisfying \eqref{limit of phi} and \eqref{generic condition}. Also (\ref{DZ2}) holds and $E(Y^{\alpha})<1$. Further
$C_{i} = \sup_{x \geq 1} {P[\zeta_{i} > x]}/{\overline{F}(x)}$ satisfies
$\sum_{i=2}^{\infty}C_{i} < \infty \quad \text{when } \alpha < 1 \quad \text{and} \quad \sum_{i=2}^{\infty}C_{i}^{\frac{1}{\alpha + \varepsilon}} < \infty \quad \text{when } \alpha \geq 1$
for some $\varepsilon > 0$. Then, for all sufficiently large $x$ uniformly in $i$, and some constant $\eta$ independent of $i$,
\begin{equation} \label{DZ2:second}
\int_{(1, \infty)}\frac{P[Z_{i} > x/v]}{P[Z_{i} > x]} G_{i}(dv) < \eta C_{i} + E[\zeta_{i}^{\alpha}].
\end{equation}
\end{lem}

\begin{proof}
We split the integral in \eqref{DZ2:second} over $(1, x]$ and $(x, \infty)$. The integral over $(x, \infty)$ is bounded as:
\begin{equation} \label{thirdpart DZ2 bound}
\int_{(x,\infty)}\frac{P[Z_{i} > x/v]}{P[Z_{i} > x]}G_{i}(dv) \leq \frac{P[\zeta_{i} > x]}{P[Z_{i} > x]} \leq C_{i} \frac{\overline{F}(x)}{P[Z_{i} > x]}.
\end{equation}
Since, from Theorem \ref{final product thm}, we know that $\overline{F}(x)/P[Z_{i} > x] \rightarrow \{E[{Y^{*}_{\theta}}^{\alpha}]\}^{-1}$, hence the right hand side of \eqref{thirdpart DZ2 bound} is bounded by a constant.

We perform integration by parts on the integral over the interval $(1, x]$ to get
\begin{equation*}
\int_{(1,x]}\frac{P[Z_{i} > x/v]}{P[Z_{i} > x]} G_{i}(dv) \leq \overline{G_{i}}(1) + \int_{(1,x]}\frac{P[\zeta_{i} > v]}{P[Z_{i} > x]}d_{v}P[Z_{i} > x/v].
\end{equation*}
The first term gets bounded by $E(\zeta_{i}^{\alpha})$ by Markov inequality. Substituting $u = \log v$ the second term is bounded by
$$C_{i} \gamma \int_{(0,\log x]}\frac{P[\log Z_{i} > u]}{P[\log Z_{i} > \log x]}d_{u}P[\log Z_{i} > \log x - u].$$

Recall that $\overline{H}(x) = P[X_{i} Y_{i} > x] = x^{-\alpha} L_{1}(x)$, where $L_{1}$ has the same representation out of (\ref{(iii)}) or (\ref{(iv)}) of Lemma \ref{slowly varying categories} as $L$. Also, $L_{1}(e^{x}) \in \mathcal{S}_{d}$. From Theorem~2.1 of \cite{kluppelberg:1989} this implies that $(\log Z_{t})^{+} \in \mathcal{S}(\alpha)$.

Therefore, there exists some $x_{2}$ large enough, independent of $t$, such that for all $x \geq x_{2}$
$$ \int_{(0,\log x]}\frac{P[\log Z_{i} > u]}{P[\log Z_{i} > \log x]}d_{u}P[\log Z_{i} > \log x - u] \leq 3E[\exp (\alpha(\log Z_{i})^{+})] \leq 3(E[Z_{i}^{\alpha}]+ 1).$$ Hence the result follows.
\end{proof}

\begin{lem} \label{DZ3:infinite sum}
Assume that $\{(X_{i}, Y_{i}), i \geq 1\}$ are i.i.d.\ random vectors with the generic random vector $(X, Y)$ following a bivariate Sarmanov distribution, satisfying \eqref{limit of phi} and \eqref{generic condition}. Also the condition (\ref{DZ3}) holds and $E(Y^{\alpha})<1$. We further have
$$\sup_{x \geq 1} \frac{P[\zeta_{i} > x]}{x^{-\alpha}P[U > \log x]} = C_{i},$$
where $\sum_{i=2}^{\infty}C_{i} < \infty \quad \text{when } \alpha < 1 \quad \text{and} \quad \sum_{i=2}^{\infty}C_{i}^{\frac{1}{\alpha + \varepsilon}} < \infty \quad \text{when } \alpha \geq 1$ for some $\varepsilon > 0$. Then we have, for all sufficiently large $x$ uniformly in $i$, and for two constants $\gamma$, $\eta$ independent of $i$,
\begin{equation} \label{DZ3:second}
\int_{(1, \infty)}\frac{P[Z_{i} > x/v]}{P[Z_{i} > x]}G_{i}(dv)  \leq \gamma E[\zeta_{i}^{\alpha}] + \eta C_{i}.
\end{equation}
\end{lem}
\begin{proof}
We split the integral in \eqref{DZ3:second} over two intervals, $(1, x]$ and $(x, \infty)$. For the integral over $(x, \infty)$, we have
$$\int_{(x,\infty)}\frac{P[Z_{i} >x/v]}{P[Z_{i} > x]}G_{i}(dv)  \leq \frac{P[\zeta_{i} > x]}{P[Z_{i} > x]}.$$
Now $\overline{H}(x) = P[X_{i} Y_{i} > x] = x^{-\alpha} L_{1}(x)$. Accordingly as $L$ is of the form (\ref{(iii)}) or (\ref{(iv)}) of Lemma \ref{slowly varying categories}, $L_{1}$ will have an analogous form with $c(x)$ replaced by $c_{1}(x)$. Thus we have, for all sufficiently large $x$ uniformly in $i$, $$\frac{P[\zeta_{i} > x]}{P[Z_{i} > x]} \leq \frac{P[\zeta_{i} > x]}{c_{1}(x)x^{-\alpha}P[U > \log x]} \leq \frac{2C_{i}}{c_{1}}.$$

For the integral over $(1, x]$, when $L$ is of the form (\ref{(iii)}) or (\ref{(iv)}),
\begin{align} \label{big integral}
\int_{(1,x]}\frac{P[Z_{i} > x/v]}{P[Z_{i} > x]}G_{i}(dv) \leq & \sup_{v \in [1,x]}\frac{c_{1}(x/v)}{c_{1}(x)}\int_{(1,x]}\frac{P[U > \log x - \log v]}{P[U > \log x]}v^{\alpha}G_{i}(dv) \nonumber\\
\leq & A \int_{(1,x]}\frac{P[U > \log x - \log v]}{P[U > \log x]}v^{\alpha}G_{i}(dv),
\end{align}
since $\lim_{x \rightarrow \infty} c_{1}(x) = c_{1}$ and hence $\sup_{v \in [1,x]}c_{1}(x/v)/c_{1}(x) < \infty$. We bound the integral  in \eqref{big integral} by using integration by parts, which gives the bound
$$\overline{G_{i}}(1) + \alpha \int_{(1,x]} \frac{v^{\alpha -1}P[U > \log x - \log v]}{P[U > \log x]}\overline{G_{i}}(v)dv + \int_{(1,x]}\frac{\overline{G_{i}}(v)v^{\alpha}}{P[U > \log x]}d_{v}P[U > \log x - \log v].$$

The first term is bounded by $E(\zeta_{i}^{\alpha})$ by Markov inequality. The second term can be dealt with as follows:
\begin{align}
\alpha \int_{(1,x]} \frac{v^{\alpha -1}P[U > \log x - \log v]}{P[U > \log x]}\overline{G_{i}}(v)dv & \leq \alpha C_{i}\int_{(1,x]} \frac{P[U > \log x - \log v]P[U > \log v]}{P[U > \log x]}\frac{1}{v}dv \nonumber\\
& < 3 \alpha C_{i} \int_{0}^{\infty}P[U>u]du \nonumber
\end{align}
for all sufficiently large $x$ uniformly in $i$. The last inequality follows from the substitution $w = \log v$ and noting that $U \in \mathcal{S}^{*}$ implies $U$ is subexponential. For the third term, we have, again for all sufficiently large $x$ uniformly in $i$,
\begin{align}
\int_{(1,x]}\frac{\overline{G_{i}}(v)v^{\alpha}}{P[U > \log x]}d_{v}P[U > \log x - \log v]
& \leq C_{i} \int_{(1,x]} \frac{P[U > \log v]}{P[U > \log x]}d_{v}P[U > \log x - \log v] \nonumber\\
& = C_{i} \frac{P[U+U' > \log x]}{P[U > \log x]} < 3C_{i}, \nonumber
\end{align}
where in the last step we use subexponentiality of $U$ and $U'$. Combining everything, the result follows.
\end{proof}

\begin{lem} \label{DZ4:infinite sum}
Assume that $\{(X_{i}, Y_{i}), i \geq 1\}$ are i.i.d.\ random vectors with the generic random vector $(X, Y)$ following a bivariate Sarmanov distribution, satisfying \eqref{limit of phi} and \eqref{generic condition}. The condition (\ref{DZ4}) also holds and $E(Y^{\alpha})<1$.
We also have $$\sup_{x \geq 1}\frac{P[\zeta_{i} > x]}{\overline{F}(x)}m(x) = C_{i} \in (0, \infty)$$ with
$\sum_{i=2}^{\infty}C_{i} < \infty \quad \text{when } \alpha < 1 \quad \text{and} \quad \sum_{i=2}^{\infty}C_{i}^{\frac{1}{\alpha + \varepsilon}} < \infty \quad \text{when } \alpha \geq 1$
for some $\varepsilon > 0$. Then for all sufficiently large $x$ uniformly in $i$, and constants $\gamma, \eta$ independent of $i$,
\begin{equation} \label{DZ4:second}
\int_{(1, \infty)}\frac{P[Z_{i} > x/v]}{P[Z_{i} > x]}G_{i}(dv) \leq \gamma E[\zeta_{i}^{\alpha}] + \eta C_{i}.
\end{equation}
\end{lem}
\begin{proof}
We split the integral in \eqref{DZ4:second} over two intervals, $(1, x]$ and $(x, \infty)$. We bound the integral over $(x, \infty)$ as follows:
$$\int_{(x,\infty)}\frac{P[Z_{i} > x/v]}{P[Z_{i} > x]}G_{i}(dv) \leq \frac{P[\zeta_{i} > x]}{\overline{F}(x)} \cdot \frac{\overline{F}(x)}{P[Z_{i} > x]}.$$

Since $\overline{F}(x)/P[Z_{i} > x]$ converges and hence bounded, and $m(x) \rightarrow \infty$,  leading to $P[\zeta_{i} > x]/\overline{F}(x) \leq C_{i}$, for all sufficiently large $x$ uniformly in $i$, thus we have, again for all sufficiently large $x$ uniformly in $i$, $\int_{(x,\infty)}\frac{P[Z_{i} > x/v]}{P[Z_{i} > x]}G_{i}(dv)$ is bounded above by a constant multiple of $C_{i}$, the constant independent of $i$.

We now consider the integral over the interval $(1, x]$ and further split it into two sub-intervals: $(1, \sqrt{x}]$ and $(\sqrt{x}, x]$ and bound them separately.

We have, for all sufficiently large $x$ uniformly in $i$,
\begin{equation*}
\int_{(1, \sqrt{x}]}\frac{P[Z_{i} > x/v]}{P[Z_{i} > x]}G_{i}(dv) \leq \sup_{u \in [\sqrt{x}, x)}\frac{L_{1}(u)}{L_{1}(x)} \int_{(1,\sqrt{x}]}v^{\alpha}G_{t}(dv)
\end{equation*}
which is bounded above by a constant multiple of $E[\zeta_{i}^{\alpha}]$, the constant free of $i$. For the integral over the subinterval $(\sqrt{x}, x]$, we integrate by parts to obtain
\begin{align} \label{split DZ4}
\int_{(\sqrt{x}, x]}\frac{P[Z_{i}>x/v]}{P[Z_{i}>x]}G_{i}(dv) &= -x^{\alpha}\frac{L_{1}(1)}{L_{1}(x)}\overline{G_{i}}(x) + x^{\alpha/2}\frac{L_{1}(\sqrt{x})}{L_{1}(x)}\overline{G_{i}}(\sqrt{x}) + \int_{(\sqrt{x}, x]}\frac{\overline{G_{i}}(v)}{L_{1}(x)}d_{v}(v^{\alpha}L_{1}(x/v)) \nonumber\\
&\leq x^{\alpha/2}\frac{L_{1}(\sqrt{x})}{L_{1}(x)}\overline{G_{i}}(\sqrt{x}) + \int_{(\sqrt{x}, x]}\frac{\overline{G_{i}}(v)}{L_{1}(x)}d_{v}(v^{\alpha}L_{1}(x/v)).
\end{align}
By Markov's inequality, for all sufficiently large $x$ uniformly in $i$, we have
$$x^{\alpha/2}\frac{L_{1}(\sqrt{x})}{L_{1}(x)}\overline{G_{i}}(\sqrt{x}) \leq \sup_{\sqrt{x} \leq y \leq \sqrt{x}}\frac{L_{1}(y)}{L_{1}(x)}E(\zeta_{i}^{\alpha}),$$
which is again bounded above by a constant multiple of $E[\zeta_{i}^{\alpha}]$, the constant free of $i$.

We now bound the second term of \eqref{split DZ4} as follows. For all sufficiently large $x$ uniformly in $i$,
\begin{align} \label{final term in DZ4}
\int_{\sqrt{x}}^{x}\frac{\overline{G_{i}}(v)}{L_{1}(x)}d_{v}(v^{\alpha}L_{1}(x/v)) &= \int_{\sqrt{x}}^{x}\frac{\overline{G_{i}}(v)}{P[X_{1} >x]}d_{v}(P[X_{1} > x/v]) \nonumber\\
&\leq \frac{C_{i}}{m(\sqrt{x})} \int_{\sqrt{x}}^{x}\frac{P[X_{1} > v]}{P[X_{1} > x]}d_{v}(P[X_{1} > x/v]) \quad \text{as } m \text{ is increasing}\nonumber\\
&\leq \frac{C_{i}}{m(\sqrt{x})} \sup_{\sqrt{x} \leq y \leq x} \frac{L(y)}{L(x)}\int_{\sqrt{x}}^{x}(\frac{x}{v})^{\alpha} d_{v}(P[X_{1} > x/v]) \nonumber\\
&= \frac{C_{i}}{m(\sqrt{x})} \sup_{\sqrt{x} \leq y \leq x} \frac{L(y)}{L(x)} \int_{1}^{\sqrt{x}}y^{\alpha}d_{y}(P[X_{1} \leq y])\nonumber,
\end{align}
which is bounded above by a multiple of $C_{i}$, the multiple free of $i$. Combining, the result follows.
\end{proof}

We summarize the consequence of all the previous results from this section in our final theorem.
\begin{thm} \label{infinite sum}
Assume that $\{(X_{i}, Y_{i}), i \geq 1\}$ are i.i.d.\ random vectors with the generic random vector $(X, Y)$ following a bivariate Sarmanov distribution, as defined in Definition \ref{sarmanov definition}, with $X \in RV_{-\alpha}$. Let $E[Y^{\alpha}] < 1, \overline{G}(x) = o(\overline{F}(x))$ and $\lim_{x \rightarrow \infty} \phi_{1}(x) = d_{1}$. Assume that one of the four (DZ) conditions holds. If one of (\ref{DZ2}), (\ref{DZ3}) and (\ref{DZ4}) is satisfied, then define
\begin{equation*}
C_{i} =
\begin{cases}
\sup_x \frac{P[\zeta_{i} > x]}{P[X_{1}>x]}, &\text{when~(\ref{DZ2}) holds,}\\[1ex]
\sup_x \frac{P[\zeta_{i} > x]}{x^{-\alpha}P[U>\log x]}, &\text{when~(\ref{DZ3}) holds,}\\[1ex]
\sup_x \frac{P[\zeta_{i} > x]}{P[X_{1}>x]} m(x), &\text{when~(\ref{DZ4}) holds,}
\end{cases}
\end{equation*}
and further assume that $$\sum_{i=2}^{\infty}C_{i} < \infty \quad \text{when } \alpha < 1 \quad \text{and} \quad \sum_{i=2}^{\infty}C_{i}^{\frac{1}{\alpha + *\varepsilon}} < \infty \quad \text{when } \alpha \geq 1$$ for some $\varepsilon > 0$. Then
$$P\left[\sup_{n \geq 1} \sum_{i=1}^{n} X_{i} \prod_{j=1}^{i} Y_{j} > x \right] \sim \frac{E[Y^{\alpha}] + \theta d_{1} E[\phi_{2}(Y) Y^{\alpha}]}{1 - E[Y^{\alpha}]} P[X_{1} > x].$$
\end{thm}

\section{Acknowledgement}
The first author acknowledges partial support by the project RARE-318984, a Marie Curie IRSES Fellowship within the 7th European Community Framework Programme. The work formed part of the Dissertation for Master of Statistics degree of the second author at Indian Statistical Institute, Kolkata, India.

\bibliographystyle{elsarticle-num}
\bibliography{mybibfile}

\end{document}